\documentclass[10pt]{article}
\usepackage{amsmath}
\usepackage{amsthm}
\usepackage{amsfonts}
\usepackage{amssymb}
\usepackage{amscd}
\usepackage{graphicx}
\usepackage{xcolor}
\usepackage{enumerate}

\font\smallsc=cmcsc10
\font\smallsl=cmsl10

\usepackage{hyperref}

\usepackage[all,knot]{xy}
\xyoption{arc} 

\newtheorem{theorem}{Theorem}

\newtheorem{lemma}[theorem]{Lemma}

\newtheorem{proposition}[theorem]{Proposition}

\newcommand{\gon}{\mathrm{gon}}

\newcommand{\Pbb}{\mathbb P}
\newcommand{\Cbb}{\mathbb C}
\newcommand{\Ocal}{\mathcal O}

\newcommand{\nlinha}[1]{{n'}^{(#1)}}

\renewcommand{\:}{\colon }

\begin{document}

\title{Gonality of curves whose normalizations are\\ one or two copies of $\mathbb P^1$}
%genus-$g$ irreducible curves with $g$ nodes and binary curves}
\author{Juliana Coelho\\{\scriptsize julianacoelhochaves@id.uff.br}
%\and Frederico Sercio\\{\scriptsize fred.feitosa@ufjf.edu.br}
}

\maketitle

\begin{abstract}
We study the gonality of curves $C$ over $\mathbb C$ whose normalization consist of one or two copies of $\mathbb P^1$. 
In the first case, $C$ is a nodal curve with $g(C)$ nodes, and in the second case $C$ is a so-called binary curve. In any case we show that 
the usual bound $\gon(C)\leq \left\lfloor\frac{g(C)+3}{2}\right\rfloor$ holds if $g(C)\geq 2$, with equality holding generically.
\end{abstract}

%\tableofcontents

\section{Introduction}

%\begin{comment}

A smooth curve $C$ is said to be $k$-gonal if it admits a degree-$k$ map to $\mathbb P^1$. 
The gonality $\gon(C)$, that is, the minimum $k$ such that $C$ is $k$-gonal, is an important numerical invariant in the study of algebraic curves. When considering smooth curves in families, one feels the need to consider singular curves as well, since smooth curves can degenerate to singular ones. 

There are a few possible ways to extend the notion  of gonality to singular curves. 
One way is to simply take the same definition as in the smooth case, as is done in \cite{sthor} and   \cite{renato}.
The drawback here is that this definition does not work very well in families, since not all curves that are limits of smooth $k$-gonal ones have maps of degree $k$ to $\Pbb^1$. 
One solution is to consider limits of maps from smooth curves to $\Pbb^1$.
This can be done using stable maps, as in \cite{cv},
or using Harris and Mumford's admissible covers \cite{HarrMum}. 
In this paper we adopt the latter point of view, as was done previously 
in \cite{paper1}, where admissible covers were used to obtain bounds on the gonality of stable curves from the gonality of its components,
and \cite{paper2}, where two-component nodal curves of gonality 2 or 3 were characterized, with the simplifying hypothesis that the components are not rational.

%In \cite{paper1} and \cite{paper2} admissible covers were used to obtain bounds on the gonality of stable curves from the gonality of its components. Moreover, in \cite{paper2}, two-components nodal curves of gonality 2 or 3 wre characterized, with the simplifying hypothesis that the components are not rational. 

It is worth mentioning other approaches to the problem.
For instance, since a map from a smooth curve to $\Pbb^1$ is given by a linear series on the curve, 
one can consider Eisenbud and Harris' limit linear series introduced in \cite{eisenharris} (see for instance \cite{em} and \cite{osser}). 
Furthermore, some authors have been working on relating the gonality of stable curves to the gonality of graphs and tropical curves (see \cite{caporaso}).

%\end{comment}

%NOT MUCH KNOWN AB GONALITY OF STABLE NON SMOOTH CURVES (REFS), 

% {\bf IN 1ST PAPER CONSTRUCTED MAPS, SHOWED STUFF ON GONALITY OF STABLE FROM GONALITY OF NORMALIZATIONS AND COMPONENTS, GOT BOUNDS, ETC (DESCRIBE RÁPIDO RESULTS OF PAPER 1) }

%IN THSI PAPER WE CONSIDER THE PROBLEM OF GENERALIZING THE WELL KNOWN FACT:

\subsection{Main Results}

In this paper we always work over the field of complex numbers $\mathbb{C}$,
and we study the gonality of nodal curves $C$ whose normalizations are one (in Subsection \ref{sectionirred}) or two (in Subsection \ref{sectionbinary}) copies of $\mathbb P^1$.
In the first case, $C$ is an irreducible curve with $g(C)$ nodes,
and in the second case, $C$ is a so-called binary curve.

Both types of curves have been extensively used as degeneration targets when proving results that hold generically,  since they are combinatorically easier to handle.  
For instance, in  \cite{gh80} the Brill-Noether theory of a general curve was studied by degenerating it to a irreducible curve of genus $g$ with $g$ nodes. 
In the same way, Prym maps we studied in \cite{colfred} via degeneration to binary curves.
Moreover, these curves were also used  as test models for results that hold for smooth curves but are yet unknown for singular ones, 
as in \cite{capobinary}, \cite{franciosi} and \cite{he}.

In Section \ref{sec:back} we review some background, and we prove some auxiliary results on admissible covers (see Subsection \ref{sec:admiss})
and maps between $\mathbb P^1$ (see Subsection \ref{sec:mapsp1}). 
In Section \ref{sec:main} we prove the main results of the paper, which are Theorems \ref{thmirred} and \ref{teo:binary}.  

Denote by $I_g$ (resp. $B_g$) the locus of irreducible curves with $g$ nodes in the moduli of genus-$g$ stable curves $\overline M_g$. 

\medskip
\noindent{\bf Theorem.} \; {\it 
Let $C$ be an irreducible curve with $g$ nodes (resp. a binary curve) 
having genus $g\geq 2$. 
Then
$$\gon(C)\leq \left\lfloor\dfrac{g+3}{2}\right\rfloor,$$
and equality holds if $C$ is general in $I_g$ (resp. $B_g$).}
\medskip

\section{Preliminaries}\label{sec:back}

\subsection{Technical background}\label{sec:back}

In this paper we always work over the field of complex numbers $\mathbb{C}$.

%%%%%%%%%%%%%%%%%%%%%%%%%%%%%%%%%%%%%%%%%%%%%%%%%%%%%%%%%%%%%%%%%%%%%%%%%%%%
%\begin{comment}

A \textit{curve} $C$ is a connected, projective and reduced scheme of dimension $1$ over $\Cbb$. 
%We denote by $C^{\text{sing}}$ the singular locus of $C$, and by $C^{\text{sm}}$ the smooth locus of $C$.
The \emph{genus} of $C$ is $g(C):=h^1(C,\Ocal_C)$.
A \textit{subcurve} $Y $ of $C$ is a reduced subscheme of pure dimension $1$, or equivalently, a reduced union of irreducible components of $C$. 
If $Y\subseteq C$ is a subcurve, we say $Y^{c}:=\overline{C\smallsetminus Y}$ is the \emph{complement} of $Y$ in $C$.

A \textit{nodal curve} $C$ is a curve with at most ordinary double points, called \emph{nodes}. 
A node is said to be \emph{separating} if there is a subcurve $Y$ of $C$ such that $Y\cap Y^c=\{n\}$. 
%The subcurves $Y$ and $Y^c$ are then said to be \emph{tails} of $C$ associated to the separating node $n$. A \emph{rational tail} is a tail of genus 0.
A \emph{rational chain}  is a nodal curve of genus 0. 
We remark that a curve  is a rational chain if and only if its nodes are all separating and its components are all rational.

Let  $g$ and $n$ be non negative integers such that $2g-2+n>0$. A \textit{$n$-pointed stable curve of genus $g$} is a curve $C$ of genus $g$ together with $n$ distinct \emph{marked points} $p_1,\ldots,p_n\in C$  such that for every  smooth rational component $E$ of $C$,  
 the number of points in the intersection   $E\cap E^{c}$ plus the number of indices $i$ such that  $p_i$ lies on $E$ is at least three. 
A \textit{stable curve} is a 0-pointed stable curve.
We denote by $\overline M_g$ the \emph{moduli space of stable curves of genus $g$}.

Moreover, we denote by  $M_{0,n}$ the \emph{moduli space of $n$-pointed  smoth curves of genus 0}.  
Such a pointed curve is simply $\mathbb P^1$ with a choice a $n$-uple of distinct points, where two $n$-uples give the same pointed curve in $M_{0,n}$ if there is an automorphism of $\mathbb P^1$ sending one to the other.
More preciselly,  if $\Delta$ is the set of $n$-uples $(p_1,\ldots,p_n)\in (\mathbb P^1)^n$ such that $n_i\neq n_j$ for all  $i\neq j$, then there is a surjective map
\begin{equation}\label{eq:M0n}
P_n:=(\mathbb P^1)^n\setminus \Delta\rightarrow M_{0,n}
\end{equation}
whose fibers have dimension 3. 
In particular, $\dim(M_{0,n})=n-3$.

\subsection{Gonality and admissible covers}\label{sec:admiss}

A smooth curve is $k$-gonal if it admits a map of degree $k$ to $\mathbb P^1$. 
A stable curve $C$ is 
$k$-gonal if it is a limit of smooth $k$-gonal curves in the moduli space $\overline M_g$ of stable curves. 
 More precisely, 
a \textit{smoothing} of a curve $C$ is a proper and flat morphism 
$f\:\mathcal{C}\rightarrow \mathrm{Spec}(\mathbb{C}[[t]])$ 
whose fibers are curves, with special fiber $C$ and
such that $\mathcal{C}$ is regular.
A stable curve $C$ is then said to be \emph{$k$-gonal} if it admits a
 smoothing $f\:\mathcal{C}\rightarrow S$ whose general fiber is
a $k$-gonal smooth curve and the special fiber is $C $.

For the sake of completeness, we remark that some authors consider any curve to be $k$-gonal if it admits a degree-$k$ map to $\mathbb P^1$. 
Altough this is a simpler definition, it does not necessarily work well in families.

The \emph{gonality} of a stable curve $C$, denoted $\gon(C)$, is the smallest $k$ such that $C$ is $k$-gonal. 
We say that $C$ is \emph{hyperelliptic} if $\gon(C)=2$. 
It is a well-known fact that all stable curves of genus $2$ are hyperelliptic.

Alternatively, $k$-gonal stable curves can be characterized in terms of admissible covers.
A \emph{$k$-sheeted admissible cover} consists of  a finite morphism $\pi:C\rightarrow B$ of degree $k$, such that $B$ and $C$ are nodal curves and:
\begin{enumerate}
\item $\pi^{-1}\left(B^{\text{sing}}\right)  =C^{\text{sing}}$;
\item $\pi$ is simply branched away from $C^{\text{sing}}$, that is, 
over each smooth point of $B$ there exists at most one point of $C$ 
where $\pi$ is ramified and this point has ramification index 2;
\item $B$ is a stable pointed curve of genus 0, when considered with its smooth points on the branch locus of $\pi$;
\item for every node $q$ of $B$ and every node $n$ of $C$ 
lying over it, the
two branches of $C$ over $n$ map to the branches of $B$ over $q$ with the same
ramification index.
\end{enumerate}

%Let $C$ be a nodal curve. A nodal curve $C^{\prime}$ is said to be \emph{stably equivalent} to $C$, denoted by $C^{\prime} \steq C$,  if $C$ can be obtained from $C^{\prime}$ by contracting to a point some of the smooth rational components of $C^{\prime}$ meeting the other components of $C^{\prime}$ in only one or two points.   We remark that, despite the name,  stable equivalence is not an equivalence relation, but is actually a partial order and, although we adhere to the usual nomenclature, we choose to adopt a more appropriate symbol for this relation.
%We remark that, despite the classical nomenclature, stable equivalence is actually a partial order and not an equivalence relation. Note that if $C'\steq C$, then  $g(C')=g(C)$ and there is a contraction map $\tau\:C'\rightarrow C$. We then say that a point $p'\in C'$ \emph{lies over} a point $p\in C$ if $\tau(p')=p$.

Let $C$ be a nodal curve. We say that a nodal curve $C^{\prime}$ is \emph{stably equivalent} to $C$ if $C$ can be obtained from $C^{\prime}$ by contracting to a point some of the smooth rational components of $C^{\prime}$ meeting the other
components of $C^{\prime}$ in only one or two points. In that case, $g(C)=g(C')$ and there is a contraction map $\tau\:C'\rightarrow C$.
We say that a point $p'\in C'$ \emph{lies over} a point $p\in C$ if $\tau(p')=p$.

We remark that, if 
$C'$ is stably equivalent to a curve $C$ then, by \cite[Lemma 3.2]{paper1},
we can consider the smooth components of $C$ as components of $C'$ as well.

The following result is a consequence of \cite[Thm. 4, p. 58]{HarrMum} and relates the notion of admissible covers to that of gonality of stable curves.

\begin{theorem}[Harris-Mumford]\label{HarrMum}
A stable curve $C$ is $k$-gonal if and only if
there exists a $k$-sheeted admissible cover $C^{\prime}\rightarrow B$, 
where $C^{\prime}$ is stably equivalent to $C$.
\end{theorem}
\begin{proof}
Cf. \cite[Thm. 3.160, p. 185]{HM}.
\end{proof}

The next two thechnical results are similar to \cite[Proposition 10]{paper2} and will be used in Subsections \ref{sectionirred} and \ref{sectionbinary}.

\begin{lemma}\label{lemairred}
Let $C$ be an  irreducible nodal curve with nodes $n_1,\ldots,n_\delta$.
Let $\pi\:C'\rightarrow B$ be an admissible cover, where $C'$ is stably equivalent to $C$. 
Let $C^\nu$ be the normalization of $C$ and let $n_j^{(1)},n_j^{(2)}\in C^\nu$ be the branches of $n_j$, for $j=1,\ldots,\delta$. 
Then $C^\nu$ is a component of $C'$ and 
$$\deg(\pi)\geq \deg(\pi|_{C^\nu})+\ell,$$
where $\ell$ is the cardinality of
$\{1\leq j\leq \delta \ | \ \pi(n_j^{(1)})\neq \pi(n_j^{(2)})\}$.
\end{lemma}
\begin{proof}
By \cite[Lemma 3.2]{paper1}, $C^\nu$ can be considered as a component of $C'$.
Set 
$B'=\pi(C^\nu)$ and let $k'$ be the degree of $\pi|_{C^\nu}\:C^\nu\rightarrow B'$.  

Now, for each $1\leq j\leq \delta$ let $W_j:=\tau^{-1}(n_j)$,
where $\tau\:C'\rightarrow C$ is the contraction map.
Then $W_j$ is a rational subcurve of $C'$ containing both $n_j^{(1)}$ and $n_j^{(2)}$
and thus $\pi(W_j)$ contains both $\pi(n_j^{(1)})$ and $\pi(n_j^{(2)})$. 

If
$\pi_1(n_j^{(1)})\neq \pi_2(n_j^{(2)})$, 
since $\pi(W_j)$ contains two distinct points $\pi(n_j^{(1)})$ and $\pi(n_j^{(2)})$ on $B'$, then it must contain the entire line $B'$.
So, for a general point $q\in B'$, $\pi^{-1}(q)$ contains at least one point on $W_j$.
Therefore $\pi^{-1}(q)$ contains $k'$ points on $C^\nu$, counted with multiplicity, and at least one point
on $W_j$ for each $1\leq j\leq \delta$ such that $\pi_1(n_j^{(1)})\neq \pi_2(n_j^{(2)})$, and the result is proven.
\end{proof}

\begin{lemma}\label{lemaB1B2}
Let $C$ be a nodal curve with irreducible smooth components $C_1$ and $C_2$ and nodes $n_1,\ldots,n_\delta$. 
For $i=1,2$ let $n_j^{(i)}$ be the branch of $n_j$ at $C_i$, for $j=1,\ldots,\delta$.
Let $\pi\:C'\rightarrow B$ be an admissible cover, where $C'$ is stably equivalent to $C$. 
\begin{enumerate}[i.]
\item If $\pi(C_1)\neq\pi(C_2)$ then $\deg(\pi)\geq \delta$.

\item If $\pi(C_1)=\pi(C_2)$, let $\ell$ be the cardinality of the set
$$\{1\leq j\leq \delta \ | \ \pi(n_j^{(1)})\neq \pi(n_j^{(2)})\}.$$
Then $\deg(\pi)\geq \deg(\pi|_{C_1})+\deg(\pi|_{C_2})+\ell$.
\end{enumerate}
\end{lemma}
\begin{proof} %MELHORA A PROP 10(i) DO ARTIGO 2 FRED \cite[Proposition 10]{paper2} 
For each $1\leq j\leq \delta$ let $W_j:=\tau^{-1}(n_j)$,
where $\tau\:C'\rightarrow C$ is the contraction map.
Then $W_j$ is a rational subcurve of $C'$ containing both $n_j^{(1)}$ and $n_j^{(2)}$
and thus $\pi(W_j)$ contains both $\pi(n_j^{(1)})$ and $\pi(n_j^{(2)})$. 
For each $i=1,2$, let $k_i$ be the degree of $\pi_i:=\pi|_{C_i}$
and set 
$B_i=\pi(C_i)$.

First we prove (i), so assume $B_1\neq B_2$.
By \cite[Lemma 9]{paper2}, either $B_1\cap B_2$ is a point or $B_0=B_1^c\cap B_2^c$ is a rational chain and $B_i\cap B_0$ is a unique point. In any case, we thus determine a unique point $q_i\in B_i$.

For each $1\leq j\leq \delta$ such that $n_j^{(i)}\neq q_i$,
the image $\pi(W_j)$ contains $B_i$.
Indeed,  $\pi(W_j)$ 
contains $\pi(n_j^{(1)})\in B_1$ and $\pi(n_j^{(2)})\in B_2$ and hence it must also contain  $q_i$. 
Since $\pi(W_j)$ contains two distinct points of $B_i$, namely $\pi(n_j^{(i)})$ and $q_i$, it must contain the entire line $B_i$.

Note that there are at most $k_i$ indexes $j$ such that $\pi(n_j^{(i)})=q_i$, 
that is, there at least $\delta-k_i$ indexes $j$ such that $\pi(n_j^{(i)})\neq q_i$.
Moreover, 
for a general point $q\in B_i$, 
$\pi^{-1}(q)$ contains $k_i$ points on $C_i$, counted with multiplicity, and 
at least one pont
on $W_j$ for each $1\leq j\leq \delta$ such that $\pi(n_j^{(i)})\neq q_i$.
Hence
$\deg(\pi)\geq k_i+\delta-k_i=\delta$.

Now we show (ii), so assume $B_1=B_2$.
Let $1\leq j\leq \delta$ be such that $\pi_1(n_j^{(1)})\neq \pi_2(n_j^{(2)})$.
Since $\pi(W_j)$ contains two distinct points $\pi_1(n_j^{(1)})$ and $\pi_2(n_j^{(2)})$ on $B_i$, then it must contain the entire line $B_i$.
So, for a general point $q\in B_i$, $\pi^{-1}(q)$ contains at least one point on $W_j$.
Therefore $\pi^{-1}(q)$ contains $k_1$ points on $C_1$ and $k_2$ points on $C_2$, counted with multiplicity, and at least one point
on $W_j$ for each $1\leq j\leq \delta$ such that $\pi_1(n_j^{(1)})\neq \pi_2(n_j^{(2)})$, and the result is proven.
\end{proof}

The following result is a slight improvement upon \cite[Proposition 3.11]{paper1}, and will be used in Subsection \ref{sectionbinary}.

\begin{proposition}\label{lemaprop311}
Let $C$ be a stable curve with $p$ smooth irreducible components $C_1,\ldots,C_p$.
For each node $n_j$ of $C$, let $C_{j(1)}$ and $C_{j(2)}$ be the irreducible components of $C$ containing $n_j$ and denote by $n_j^{(1)}$ and $n_j^{(2)}$ the branches of $n_j$ on $C_{j(1)}$ and $C_{j(2)}$, respectively. 
Assume that there is a degree-$k_s$ map $\pi_s\:C_s\rightarrow \mathbb P^1$ for each component $C_s$ of $C$ such that 
$\pi_{j(1)}(n_j^{(1)}) = \pi_{j(2)}(n_j^{(2)})$ for every node $n_j$.
Then $C$ is $(k_1+\ldots+k_p)$-gonal.
\end{proposition}
\begin{proof} %PROP 3.11 de \cite{paper1} sem a hipotese 3
Let $n_1,\ldots,n_\delta$  be the nodes of $C$ and 
%Let $q_j=\pi_{j(1)}(n_j^{(1)})= \pi_{j(2)}(n_j^{(2)})$.
let 
$$Q=\{\pi_{j(1)}(n_j^{(1)})\ | \ j=1,\ldots,\delta\}.$$
%We write $Q=\{q_1,\ldots,q_\delta\}$.
%For each $q\in Q$, let $J(q)=\{j \ | \ q_j=q\}$ and fix one $j(q)\in J(q)$. 
For each $q\in Q$, 
fix one $1\leq j(q)\leq \delta$
such that $\pi_{j(q)(1)}(n_{j(q)}^{(1)})=q$.
%Note that since $C$ is connected, we may choose the $j(q)$ such that each component of $C$ contains at least one node of the form $n_{j(q)}$.  NOT TRUE

Let 
$J=\{1,\ldots,\delta\}\setminus \{j(q)\ | \ q\in Q\}$
and let $C_J$ be the curve
obtained by normalizing $C$ at the nodes $n_j$ such that 
$j\in J$.
Note that $C_J$ may be disconnected. 
However, 
applying \cite[Proposition 3.11]{paper1} to $C_1,\ldots, C_p$ and the nodes $n_{j(q)}$ for $q\in Q$, we obtain a finite map $\pi_J'\:C_J'\rightarrow B$ of degree $k_1+\ldots+k_p$, where $C_J'$ is stably equivalent to $C_J$, satisfying conditions 1 to 4 of the definition of an admissible covering.
By construction, $\pi'_J|_{C_s}=\pi_s$ for every component $C_s$ of $C$.

%Applying \cite[Proposition 3.11]{paper1} to $C_1,\ldots, C_p$ and the nodes $n_{j(q)}$ for $q\in Q$, we obtain an admissible covering $\pi_J'\:C_J'\rightarrow B$ of degree $k_1+\ldots+k_p$, where $C_J'$ is stably equivalent to $C_J$. By construction, $\pi'_J|_{C_s}=\pi_s$ for every component $C_s$ of $C$.

Furthermore, it is easy to see that \cite[Proposition 3.4]{paper1} can be applied to a map such as $\pi_J$, even if $C_J'$ is not connected. 
We thus proceed in the following manner.
%We now apply \cite[Theorem 3.4]{paper1} to the remaining nodes of $C$ in the following manner. 
Write $J=\{j_1,\ldots,j_r\}$ and
for each $\ell=1,\ldots,r$ let $C(\ell)$ be the curve obtained by normalizing $C$ at the nodes 
$n_{j_\ell},\ldots,n_{j_r}$, and  let $C({r+1})=C$.
We will show by induction that for every $1\leq\ell\leq r+1$
there is a finite map $\pi'_\ell\:C'(\ell)\rightarrow B_\ell$ 
of degree $k_1+\ldots+k_p$ satisfying conditions 1 to 4 of the definition of an admissible covering, and 
such that $C'(\ell)$ is stably equivalent to $C(\ell)$ and $\pi'_\ell|_{C_s}=\pi_s$ for every component $C_s$ of $C$. 
Note that $\pi'_J$ satisfies the statement for $\ell=1$.

Assume that the statement holds for $\ell-1$.
Choose points  
$\nlinha{1}_{j_\ell}$ and $\nlinha{2}_{j_\ell}$ on $C'(\ell-1)$ lying over 
$n_{j_\ell}^{(1)}$ and $n_{j_\ell}^{(2)}$, respectively, such that
$\pi'_{\ell-1}(\nlinha{1}_{j_\ell})=\pi'_{\ell-1}(\nlinha{2}_{j_\ell})$. 
Then \cite[Theorem 3.4 (b)]{paper1} gives the required admissible cover $\pi'_\ell$ as stated.

The case $\ell=r+1$ shows that $C$ is $(k_1+\ldots+k_p)$-gonal.
\end{proof}

\subsection{Maps between $\mathbb P^1$}\label{sec:mapsp1}

In the following section we will study the gonality of nodal curves whose nornalizations are one or two copies of $\mathbb P^1$. For that end, we will first examine endomorphisms of $\mathbb P^1$. 

Recall that set of non-constant morphisms $\varphi\:\mathbb P^1 \rightarrow \mathbb P^1$ of degree at most $k$ can be identified with the Grassmanian $\mathcal G_k:=Gr(2,H^0(\mathbb P^1,\Ocal_{\mathbb P^1}(k)))$ of $2$-dimensional subspaces of $H^0(\mathbb P^1,\Ocal_{\mathbb P^1}(k))$.
Indeed, such a morphism corresponds to a linear system of degree $k$ and (projective) dimension 1, which, up to change of projective coordinates, corresponds to
a linear subspaces of $H^0(\mathbb P^1,\Ocal_{\mathbb P^1}(k))$ of dimension 2. 
Since $H^0(\mathbb P^1,\Ocal_{\mathbb P^1}(k))$ is a vector space of dimension $k+1$, we get in particular that $\dim(\mathcal G_k)=2k-2$.

\begin{lemma}\label{lemamapsP1}
Let $k,\ell\geq 1$ be integers such that $\ell\leq 2k-2$ and 
let  $n^{(i)}_1,\ldots,n^{(i)}_\ell\in \mathbb P^1$, for $i=1,2$.
%Let $m\geq \lceil \frac{\delta-1}2\rceil$ be an integer.
\begin{enumerate}[i.]
\item There exists $\varphi\in \mathcal G_k$ such that 
$\varphi(n_j^{(1)})=n_j^{(2)}$ for every $j=1,\ldots,\ell$.

\item There exists $\varphi\in \mathcal G_k$ such that 
$\varphi(n_j^{(1)})=\varphi(n_j^{(2)})$ for every $j=1,\ldots,\ell$.
\end{enumerate}
\end{lemma}
\begin{proof}
For $j=1,\ldots,\ell$ we let
$V_j$ (resp. $W_j$) be the set of $\varphi\in\mathcal G_k$ such that 
$\varphi(n_j^{(1)})=n_j^{(2)}$ (resp. $\varphi(n_j^{(1)})=\varphi(n_j^{(2)})$). 
Then $V_j$ (resp. $W_j$) is a hyperplane (resp. quadric hypersurface) in $\mathcal G_k$. 
Since $\ell\leq \dim(\mathcal G_k)$,  the intersection $V_1\cap\cdots\cap V_\ell$ 
(resp. $W_1\cap\cdots\cap W_\ell$) 
is non-empty and any point on this intersection gives the desired $\varphi$, thus showing (i) (resp. (ii)).
\end{proof}

\section{Results on gonality}\label{sec:main}

%In \cite{paper1} and \cite{paper2} admissible covers were used to obtain bounds on the gonality of stable curves from the gonality of its components. Moreover, in \cite{paper2}, two-components nodal curves of gonality 2 or 3 were characterized, with the simplifying hypothesis that the components are not rational. 

In \cite{paper2}, the gonality of two-component nodal curves is studied, focusing on the case where both components are non-rational. In this work we study the case of nodal curves whose normalization are one (Subsection \ref{sectionirred}) or two (Subsection \ref{sectionbinary}) copies of $\mathbb P^1$.

\subsection{Irreducible curves of genus $g$ with $g$ nodes}\label{sectionirred}

In this subsection we study the gonality of irreducible curves $C$ having $g(C)$ nodes. 
We note that the normalization of such a curve is $\mathbb P^1$. 
Indeed, the curve $C$ is obtained from $\mathbb P^1$ by identifying $g(C)$ pairs of distinct points on $\mathbb P^1$.

\begin{center}
\[
\xy
(0,50)*{}="A";
(0,0)*{}="B";
"A"; "B" **\crv{(-8,40) & (10,30)& (13,35) & (10,40) & (-8,30) & (10,20)& (13,25) & (10,30) & (-8,20) & (10,10)& (13,15) & (10,20) & (-8,10)};
(2,37)*{\mbox{\footnotesize $n_1$}};
(2,27)*{\mbox{\footnotesize $n_2$}};
(2,17)*{\mbox{\footnotesize $n_3$}};
(2,-2)*{\mbox{\footnotesize $C$}};
(-30,50)*{}="C";
(-30,0)*{}="D";
"C"; "D" **\dir{-};
(-30,43)*{-};
(-33,44)*{\mbox{\footnotesize $n_1^{(1)}$}};
(-30,36)*{-};
(-33,37)*{\mbox{\footnotesize $n_1^{(2)}$}};
(-30,28)*{-};
(-33,29)*{\mbox{\footnotesize $n_2^{(1)}$}};
(-30,21)*{-};
(-33,22)*{\mbox{\footnotesize $n_2^{(2)}$}};
(-30,13)*{-};
(-33,14)*{\mbox{\footnotesize $n_3^{(1)}$}};
(-30,6)*{-};
(-33,7)*{\mbox{\footnotesize $n_3^{(2)}$}};
{\ar^{\nu} (-20,25); (-10,25)};
(-28,-2)*{\mbox{\footnotesize $C^{\nu}=\mathbb{P}^1$}};
\endxy
\]
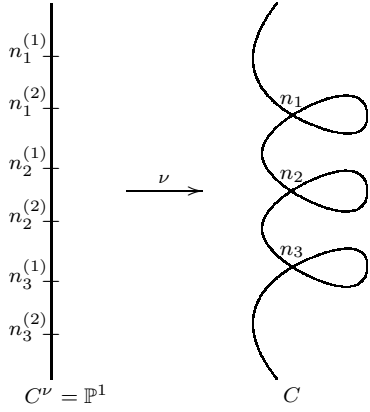
An irreducible curve of genus 3 with 3 nodes.
\end{center}

\smallskip

Denote by 
$I_g$ the  locus  in $\overline M_g$ of irreducible curves of genus $g$ with $g$ nodes.

\begin{theorem}\label{thmirred}
Let $C$ be an irreducible curve with $g(C)$ nodes. 
Then 
$$\gon(C)\leq \left\lceil\dfrac{g(C)+3}{2}\right\rceil$$
and equality holds if $C$ is general in $I_g$. 
\end{theorem}
\begin{proof}
Let $C$ be an irreducible curve of genus $g=g(C)$ and nodes $n_1,\ldots,n_g$.
Then  the branches of the nodes give two $g$-uples of points 
$\underline n^{(i)}=(n_1^{(i)}, \ldots, n_g^{(i)})\in (\mathbb P^1)^{g}$, 
such that 
\begin{equation}\label{eqdiagonal}
n_j^{(i)}\neq n_{j'}^{(i')} \;\text{ unless }\;(i,j)=(i',j')
\end{equation}
for $i=1,2$ and $j=1,\ldots,g$.

By Lemma \ref{lemamapsP1} (ii), there exists
a non-constant morphism $\varphi\:\mathbb P^1\rightarrow \mathbb P^1$ of degree at most $k=\lfloor \frac{g+2}2\rfloor$
such that 
$\varphi(n_j^{(1)})=\varphi(n_j^{(2)})$ for every $j=1,\ldots,g$.
Hence, by \cite[Corollary 3.5]{paper1}, $C$ is $k$-gonal and 
$\gon(C)\leq k=\lfloor \frac{g+2}2\rfloor=\lceil\frac{g+3}2\rceil,$
showing the first statement.

To show the second statement, we 
note that $P:=P_{2g}$ is 
the set of $(\underline n^{(1)},\underline n^{(2)})\in (\mathbb P^1)^g\times (\mathbb P^1)^g$ satisfying \eqref{eqdiagonal} (see Subsection \ref{sec:back}).
Since any $(\underline n^{(1)},\underline n^{(2)})\in P$ gives an irreducible curve of genus $g$ with $g$ nodes by identifying $n_j^{(1)}$ and $n_j^{(2)}$ for every $j=1,\ldots,g$, then
there is a surjective map  
$\Phi\:P\rightarrow I_g$.

For each $k,r\in \mathbb N$,  denote by $P(k,r)$ the set of the pairs 
$(\underline n^{(1)},\underline n^{(2)})$  in $P$ such that 
there exists $I\subset \{1,\ldots,g\}$ of cardinality $r$ and 
a non-constant map $\varphi\:\mathbb P^1\rightarrow \mathbb P^1$ of degree at most $k$ satisfying
$\varphi(n_j^{(1)})=\varphi(n_j^{(2)})$ for all $j\in I$.
Then $P(k,r)\subset P(k+1,r)$, 
$P(k,r)$ is closed in $P$ 
and, by Lemma \ref{lemamapsP1} (ii), we have 
$P(\lceil \frac{r+3}{2}\rceil,r)=P$ for every $1\leq r\leq g$.

Set 
$$U:=P\setminus
\bigcup_{r=1}^g 
P\left(\left\lceil \frac{r+3}{2}\right\rceil-1,r\right).$$
Since $U$ is open and dense in $P$,
its image $\Phi(U)$ contains an open an dense subscheme of $I_g$.
Let 
$C$ be an irreducible curve with $g(C)=g$ nodes corresponding to a point $(\underline n^{(1)},\underline n^{(2)})\in U$.   
We will show that
$\gon(C)=\lceil\frac{g+3}{2}\rceil$. 

Let 
$\pi\: C'\rightarrow B$ be an admissible cover of degree $k:=\gon(C)$, where $C'$ is stably equivalent to $C$.
Restricting  $\pi$ to the normalization $C^{\nu}=\mathbb P^1$ of $C$, we obtain a non-constant map $\varphi:=\pi|_{C^{\nu}}\:\mathbb P^1\rightarrow \mathbb P^1$ 
 of degree $k':=\deg(\varphi)$. 
Let $J$ be the set of $j\in\{1,\ldots,g\}$ such that 
$\varphi(n_j^{(1)})\neq \varphi(n_j^{(2)})$ and let $\ell=|J|$.
By Lemma \ref{lemairred}, we have $k\geq k'+\ell$.

On the other hand, let $\gamma$  be the cardinality of $I:=\{1,\ldots,g\}\setminus J$, then $\ell+\gamma=g$.
Since $(\underline n^{(1)},\underline n^{(2)})\in U$,
from the existence of the degree-$k'$ map $\varphi$ such that $\varphi(n_j^{(1)})=\varphi(n_j^{(2)})$ for every $j\in I$, we get that 
$k'\geq \lceil \frac{\gamma+3}{2}\rceil=\lceil \frac{g-\ell+3}{2}\rceil$.
Since $k\leq \lceil \frac{g+3}{2}\rceil$, then
$$\left\lceil \frac{g-\ell+3}{2}\right\rceil+\ell
\leq k'+\ell
\leq k
\leq \left\lceil \frac{g+3}{2}\right\rceil,$$
which implies that $\ell=0$ and we must have $k= \lceil \frac{g+3}{2}\rceil$.
%HAS TO define what general mean, and PROBABLY HAS TO DO WITH EXPECTED DIMENSION OF THE INTERSECTION
\end{proof}

\subsection{Binary curves}\label{sectionbinary}

A binary curve $C$ is a nodal curve made of two smooth rational components meeting at $g+1$ points. Note that the genus of $C$ is $g(C)=g$ and its normalization is two disjoint copies of $\mathbb P^1$.

More preciselly,  if $n^{(i)}_1,\ldots,n^{(i)}_{g+1}\in \mathbb P^1$ are distinct points, 
we set $\underline n^{(i)}=(n^{(i)}_1,\ldots,n^{(i)}_{g+1})$, 
for $i=1,2$,
and 
we denote by $C(\underline n^{(1)},\underline n^{(2)})$ the binary curve obtained by gluing two copies of $\mathbb P^1$ at $n^{(1)}_j$ and $n^{(2)}_j$ for every $j$.

\begin{center}
\[
\xy
(5,50)*{}="A";
(-5,0)*{}="B";
"A"; "B" **\crv{(-30,25) & (30,25)};
(-5,50)*{}="C";
(5,0)*{}="D";
"C"; "D" **\crv{(30,25) & (-30,25)};
(3,46)*{\mbox{\footnotesize $n_1$}};
(3,25)*{\mbox{\footnotesize $n_2$}};
(3,3.5)*{\mbox{\footnotesize $n_3$}};
(0,-5)*{\mbox{\footnotesize $C$}};
(-30,50)*{}="E";
(-30,0)*{}="F";
"E"; "F" **\dir{-};
(-30,44)*{-};
(-33,45)*{\mbox{\footnotesize $n_1^{(2)}$}};
(-30,25)*{-};
(-33,26)*{\mbox{\footnotesize $n_2^{(2)}$}};
(-30,5)*{-};
(-33,6)*{\mbox{\footnotesize $n_3^{(2)}$}};
(-30,-5)*{\mbox{\footnotesize $\mathbb{P}^1$}};
(-45,50)*{}="G";
(-45,0)*{}="H";
"G"; "H" **\dir{-};
(-45,44)*{-};
(-48,45)*{\mbox{\footnotesize $n_1^{(1)}$}};
(-45,25)*{-};
(-48,26)*{\mbox{\footnotesize $n_2^{(1)}$}};
(-45,5)*{-};
(-48,6)*{\mbox{\footnotesize $n_3^{(1)}$}};
{\ar^{\nu} (-20,25); (-10,25)};
(-45,-5)*{\mbox{\footnotesize $\mathbb{P}^1$}};
\endxy
\]
A binary curve of genus 2.
\end{center}

\smallskip

Two binary curves $C(\underline n^{(1)},\underline n^{(2)})$  and 
$C(\underline m^{(1)},\underline m^{(2)})$ 
are isomorphic if there are isomorphisms $\psi_i$ of $\mathbb P^1$ such that $\psi_i(n^{(i)}_j)=m^{(i)}_j$ for every $j=1,\ldots,n$ and $i=1,2$.
Hence, if $B_g\subset \overline M_g$ is the locus of binary curves of genus $g\geq 2$,
there exists a surjective morphism with finite fibers
\begin{equation}\label{eq:Bg}
M_{0,g+1}\times M_{0,g+1}\rightarrow B_g.
\end{equation}

\begin{proposition}\label{prop:binaryhyperel}
Let $C=C(\underline n^{(1)},\underline n^{(2)})$  be a binary curve of genus $g\geq 2$.
Then $C$ is hyperelliptic if and only if there exists an automorphism $\psi\colon \mathbb P^1\rightarrow \mathbb P^1$ such that $\psi(n_j^{(1)})=n_j^{(2)}$ for every $j=1,\ldots,\delta$.
\end{proposition}
\begin{proof}
First assume there is such an automorphism. 
Set $C_1:=\mathbb P^1=:C_2$ and $\pi_1:=\psi$  and let $\pi_2$ be the identity on $\mathbb P^1$. 
Then, for $i=1,2$, 
$\pi_i\:C_i\rightarrow \mathbb P^1$ are maps of degree 1 that  satisfy the conditions of Proposition \ref{lemaprop311}, which implies that $C$ is 2-gonal. Since $C$ is not rational, we have $\gon(C)=2$.

Now assume that $C$ is hyperelliptic and let $\pi\:C'\rightarrow B$ be an admissible cover of degree 2 where $C'$ is stably equivalent to $C$. 
For $i=1,2$ let $k_i$ be the degree of $\pi_i:=\pi|_{C_i}$, 
then $k_i\leq \deg(\pi)=2$.

Let $B_i:=\pi_i(C_i)$ for $i=1,2$. 
If $B_1\neq B_2$, by Lemma \ref{lemaB1B2}, we have $3\leq g+1=\delta\leq\deg(\pi)=2$, a contradiction.
%Let $B_i:=\pi_i(C_i)$ for $i=1,2$.  First we show that $B_1=B_2$. Indeed, if $B_1\neq B_2$ then by \cite[Proposition 10]{paper2} we have  $$2=\deg(\pi)\geq k_i+|Q_i|-1,$$ where  $Q_i=\{\pi_i(n_j^{(i)})\ | \ j=1,\ldots,\delta\}$, for each $i=1,2$. Now, there are only two possibilities for $k_i$.
%If $k_i=1$ then $2\geq |Q_i|$ but, since $\pi_i$ is an isomorphism, we have $|Q_i|=\delta=g+1\geq3$, a contradiction. On the other hand, if $k_i=2 $ then $|Q_i|=1$. But then the preimage under $\pi_i$ of the single point of $Q_i$ consists of all the $n_j^{(i)}$ for $j=1,\ldots,\delta$, which is impossible since $\deg(\pi_i)=2$ and $\delta=g+1\geq 3$. 
Then $B_1=B_2$ and, 
in the notation of Lemma \ref{lemaB1B2}, we have 
$2=\deg(\pi)\geq k_1+k_2+\ell$.
Hence $k_1=k_2=1$, $\ell=0$ and $\psi:=\pi_2^{-1}\circ\pi_1$ is the required automorphism of $\mathbb P^1$. 
\end{proof}

In particular, 
since there is always an automorphism of $\mathbb P^1$ taking any set of three distinct points to another one, we see that
the previous result 
recovers, in the case of binary curves, the already know fact that all stable curves of genus 2 are hyperelliptic.

The first part of the following result was already proved in 
\cite[Lemma 5.3]{franciosi},
with the same argument. We include a proof here, for the sake of completeness.

\begin{theorem}\label{teo:binary}
Let $C$  be a binary curve of genus $g(C)\geq 2$.
%with $\delta\geq3$ nodes.
Then %$C$ is $\left\lfloor\dfrac{g(C)+3}{2}\right\rfloor$-gonal. In particular 
$$\gon(C)\leq \left\lfloor\dfrac{g(C)+3}{2}\right\rfloor$$
and equality holds if $C$ is general in $B_g$. 
\end{theorem}
\begin{proof}
Set $C=C(\underline n^{(1)},\underline n^{(2)})$ and set $g=g(C)$.
By Lemma \ref{lemamapsP1} (i) with $\ell=g-2$, there exists a 
morphism $\varphi\:\mathbb P^1\rightarrow \mathbb P^1$ of degree $k=\lceil \frac{g}2\rceil$
such that 
$\varphi(n_j^{(1)})=n_j^{(2)}$ for every $j=1,\ldots,g-2$.

Now note that $C(\underline n^{(1)},\underline n^{(2)})$ is isomorphic to the binary curve 
$C(\underline n^{(1)},\underline m^{(2)})$ where $m^{(2)}_j=n^{(2)}_j$ for $j=1,\ldots,g-2$ and 
$m^{(2)}_j=\varphi(n^{(1)}_j)$ for $j=g-1,g,g+1$. Hence we may assume that $\varphi(n_j^{(1)})=n_j^{(2)}$ for every $j=1,\ldots,g+1$.
Therefore, by Proposition \ref{lemaprop311} applied to $\varphi$ and the identity of $\mathbb P^1$, $C$ is $(k+1)$-gonal and 
$$\gon(C)\leq k+1 = \left\lceil \frac{g+2}2\right\rceil
=\left\lfloor \frac{g+3}2\right\rfloor
.$$

Now we prove the second statement. 
As in the proof of Theorem \ref{thmirred}, 
for each $k, \gamma\in \mathbb N$,  we denote by $P(k,\gamma)$ the set of the pairs 
$(\underline n^{(1)},\underline n^{(2)})$  in $P_{g+1}\times P_{g+1}$ (see Subsection \ref{sec:back}) such that 
there exists $I\subset \{1,\ldots,g+1\}$ of cardinality $\gamma$ and, for $i=1,2$, a
map $\psi_i\:\mathbb P^1\rightarrow \mathbb P^1$ satisfying 
$\psi_1(n_j^{(1)})=\psi_2(n_j^{(2)})$ for all $j\in I$
and such that $\deg(\psi_1)+\deg(\psi_2)$ is at most $k$.
Then $P(k,\gamma)\subset P(k+1,\gamma)$, 
$P(k,\gamma)$ is closed in $P_{g+1}\times P_{g+1}$ 
and, by Lemma \ref{lemamapsP1} (i), we have 
$P(\lfloor \frac{\gamma+2}{2}\rfloor,\gamma)=P$ for every $\gamma\geq 4$.
%since  $\lceil \frac{\gamma-1}{2}\rceil+1 = \lceil \frac{\gamma+1}{2}\rceil =\lfloor \frac{\gamma+2}{2}\rfloor$.

Set 
$$U:=P\setminus
\bigcup_{\gamma=1}^{g+1} 
P\left(\left\lfloor \frac{\gamma+2}{2}\right\rfloor-1,\gamma\right)$$
and note that $U$ is open and dense in $P_{g+1}\times P_{g+1}$.
If $\Phi\:P_{g+1}\times P_{g+1}\rightarrow M_{0,g+1}\times M_{0,g+1} \rightarrow B_g
$ is the composition of the maps \eqref{eq:M0n} and \eqref{eq:Bg}, then 
the image $\Phi(U)$ contains an open dense subscheme of $B_g$.
Let 
$C$ be a binary curve of genus $g$ corresponding to   a point $(\underline n^{(1)},\underline n^{(2)})\in U$.
We will show that
$\gon(C)=\lfloor\frac{g+3}{2}\rfloor$.

Let 
$\pi\: C'\rightarrow B$ be an admissible cover of degree $k:=\gon(C)$, where $C'$ is stably equivalent to $C$. 
By the first part of the result, we have $k\leq \lfloor\frac{g+3}{2}\rfloor< g+1.$
Restricting  $\pi$ to the components $C_i=\mathbb P^1$ of $C$, we obtain maps 
$\psi_i:=\pi|_{C_i}\:\mathbb P^1\rightarrow \mathbb P^1$ 
 of degree $k_i:=\deg(\psi_i)$. 
Let $J$ be the set of $j\in\{1,\ldots,g+1\}$ such that 
$\psi_1(n_j^{(1)})\neq \psi_2(n_j^{(2)})$ and let $\ell=|J|$.
By Lemma \ref{lemaB1B2},  we have $\pi(C_1)= \pi(C_2)$ and thus $k\geq k_1+k_2+\ell$.

On the other hand, let $\gamma$  be the cardinality of $I:=\{1,\ldots,g+1\}\setminus J$, then $\gamma=g+1-\ell$.
Since $(\underline n^{(1)},\underline n^{(2)})\in U$,
from the existence of $\psi_1$ and $\psi_2$ such that $\psi_1(n_j^{(1)})=\psi_2(n_j^{(2)})$ for every $j\in I$, we get that $k_1+k_2\geq 
\lfloor \frac{\gamma+2}{2}\rfloor=
\lfloor \frac{g+3-\ell}{2}\rfloor$.
Since by Proposition \ref{lemagraum} we have $k\leq \lfloor \frac{g+3}{2}\rfloor$, then
$$\left\lfloor \frac{g+3+\ell}{2}\right\rfloor
= \left\lfloor \frac{g+3-\ell}{2}\right\rfloor+\ell
\leq k_1+k_2+\ell
\leq k
\leq \left\lfloor \frac{g+3}{2}\right\rfloor
$$
which implies that $\ell=0$.
Thus $\gamma=g+1$ and $k= \lfloor \frac{g+3}{2}\rfloor$.
\end{proof}

%%%%%%%%%%%%%%%%%%%%%%%%%%%%%%%%%%%%%%%%%%%%%%%%%%%%%%%%%%%%%%%%%%%%%%%%%%%%%%%%%%%%%%%%%%%%%%%
%%%%%%%%%%%%%%%%%%%%%%%%%%%%%%%%%%%%%%%%%%%%%%%%%%%%%%%%%%%%%%%%%%%%%%%%%%%%%%%%%%%%%%%%%%%%%%%
%%%%%%%%%%%%%%%%%%%%%%%%%%%%%%%%%%%%%%%%%%%%%%%%%%%%%%%%%%%%%%%%%%%%%%%%%%%%%%%%%%%%%%%%%%%%%%%

\bibliographystyle{aalpha}
%\bibliography{acompat,articles}

\bigskip
\noindent{\smallsc Juliana Coelho, Universidade Federal Fluminense (UFF), 
Instituto de Matem\'atica e Estat\'istica - 
 Rua Prof. Marcos Waldemar de Freitas Reis, S/N, 
 Gragoat\'a, 24210-201 - Niter\'oi -  RJ,  Brasil}\\
{\smallsl E-mail address: \small\verb?julianacoelhochaves@id.uff.br?}

\end{document}